\documentclass[amstex,12pt,russian,amssymb]{article}

\usepackage{mathtext}
\usepackage[cp1251]{inputenc}
\usepackage[T2A]{fontenc}
\usepackage[russian]{babel}
\usepackage[dvips]{graphicx}
\usepackage{amsmath}
\usepackage{amssymb}
\usepackage{amsxtra}
\usepackage{latexsym}
\usepackage{ifthen}

\begin{document}

\newcounter{lemma}
\newcommand{\lemma}{\par \refstepcounter{lemma}%
{\bf Лема \arabic{lemma}.}}

\newcounter{corollary}
\newcommand{\corollary}{\par \refstepcounter{corollary}%
{\bf Наслідок \arabic{corollary}.}}

\newcounter{remark}
\newcommand{\remark}{\par \refstepcounter{remark}%
{\bf Зауваження \arabic{remark}.}}

\newcounter{theorem}
\newcommand{\theorem}{\par \refstepcounter{theorem}%
{\bf Теорема \arabic{theorem}.}}

\newcounter{proposition}
\newcommand{\proposition}{\par \refstepcounter{proposition}%
{\bf Твердження \arabic{proposition}.}}

\newcounter{example}
\newcommand{\example}{\par \refstepcounter{example}%
{\bf Приклад \arabic{example}.}}

\renewcommand{\refname}{\centerline{\bf Список літератури}}

\renewcommand{\figurename}{Мал.}

\newcommand{\proof}{{\it Доведення.\,\,}}

\noindent УДК 517.5

\medskip\medskip
{\bf А.П.~Довгопятый} (Житомирский государственный университет имени
Ивана Франко)

{\bf Е.А.~Севостьянов} (Житомирский государственный университет
имени Ивана Франко; Институт прикладной математики и механики НАН
Украины, г.~Славянск)

\medskip\medskip
{\bf О.П.~Довгопятий} (Житомирський державний університет імені
Івана Фран\-ка)

{\bf Є.О.~Севостьянов} (Житомирський державний університет імені
Івана Фран\-ка; Інститут прикладної математики і механіки НАН
України, м.~Слов'янськ)

\medskip\medskip
{\bf O.P.~Dovhopiatyi} (Zhytomyr Ivan Franko State University)

{\bf E.A.~Sevost'yanov} (Zhytomyr Ivan Franko State University;
Institute of Applied Ma\-the\-ma\-tics and Mechanics of NAS of
Ukraine, Slov'yans'k)

\medskip
{\bf О существовании решений квазилинейных уравнений Бельтрами с
двумя характеристиками}

{\bf Про існування розв'язків квазілінійних рівнянь Бельтрамі с
двома характеристиками}

{\bf On the existence of solutions of quasilinear Beltrami equations
with two characteristics}

\medskip\medskip
Изучаются уравнения типа Бельтрами с двумя заданными комплексными
характеристиками. При определённых условиях на комплексные
коэффициенты получены теоремы о существовании гомеоморфных
$ACL$-решений этого уравнения. Кроме того, при некоторых
относительно слабых условиях получены теоремы о существовании
соответствующих непрерывных $ACL$-решений, которые являются
логарифмически гёльдеровыми в заданной области.

\medskip\medskip
Вивчаються рівняння типу Бельтрамі з двома заданими комплексними
характеристиками. За певних умов на комплексні коефіцієнти отримано
теореми про існування гомеоморфних $ACL$-розв'язків цього рівняння.
Крім того, за деяких відносно слабких умов отримано теореми про
існування відповідних неперервних $ACL$-розв'язків, які є
логарифмічно гельдеровими в заданій області.

\medskip\medskip
We study Beltrami-type equations with two given complex
characteristics. Under certain conditions on the complex
coefficients, we obtained theorems on the existence of homeomorphic
$ ACL $ -solutions of this equation. In addition, for some
relatively weak conditions, we obtained theorems on the existence of
the corresponding continuous $ ACL $ -solutions that are loga
-rith\-mi\-cally H\"{o}lder in a given domain.

\newpage
{\bf 1. Вступ.} Відносно нещодавно були отримані результати про
існування розв'язків виродженого рівняння Бельтрамі з двома
характеристиками, див., напр., \cite{BGR$_1$}--\cite{BGR$_3$}
і~\cite[розділ~9]{GRSY}. Відзначимо, що вказані результати переважно
стосуються випадку, коли максимальна дилатація рівняння має
скінченне середнє коливання в кожній точці, або задовольняє певну
умову розбіжності інтегралу. В даній статті ми розглянемо аналогічні
квазілінійні рівняння, тобто, коли відповідні коефіцієнти можуть
залежати від його розв'язку. Вказана ситуація достатньо повно
розглянута в роботі другого автора, але лише в випадку одного
комплексного коефіцієнта (\cite{Sev}). Окремо будуть розглянуті
умови, які забезпечують наявність не гомеоморфних, а просто
неперервних розв'язків. Слід зауважити, що ці розв'язки мають вищу
степінь гладкости і є логарифмічно неперервними за Гельдером (див.
відповідні публікації \cite{Sev$_1$} і \cite{SevSkv} для рівнянь з
одним коефіцієнтом).

\medskip
Перейдемо до означень. Скрізь далі відображення $f:D\rightarrow
{\Bbb C}$ області $D\subset{\Bbb C}$ вважається таким, що {\it
зберігає орієнтацію,} тобто, якщо $f$ -- відкрите дискретне
відображення і $z\in D$ -- яка-небудь його точка диференційовності,
то {\it якобіан} цього відображення в точці $z$ невід'ємний (див.,
напр., \cite[лема~2.14]{MRV$_1$}). Для комплекснозначної функції
$f:D\rightarrow {\Bbb C},$ заданої в області $D\subset {\Bbb C},$ що
має частинні похідні по $x$ і $y$ при майже всіх $z=x + iy,$
покладемо $f_{\overline{z}} = (f_x + if_y)/2$ і $f_z = (f_x -
if_y)/2.$ Будемо говорити, що функція
$\nu=\nu\left(z,w\right):D\times{\Bbb C} \rightarrow {\Bbb D}$
задовольняє {\it умову Каратеодорі,} якщо $\nu$ вимірна по $z\in D$
при кожному фіксованому $w\in{\Bbb C}$ і неперервна по $w\in{\Bbb
C}$ при майже всіх $z\in D.$ Нехай функції $\mu=\mu(z, w)$ і
$\nu=\nu(z, w)$  задовольняють умову Каратеодорі і, крім того,
$|\mu(z, w)|+|\nu(z, w)|<1$ при всіх $w\in {\Bbb C}$ і майже всіх
$z\in D.$ Покладемо
\begin{equation}\label{eq1}
K_{\mu, \nu}(z,w)=\quad\frac{1+|\mu (z, w)|+|\nu(z, w)|}{1-|\mu\,(z,
w)|-|\nu(z, w)|}\,.
\end{equation}
Зауважимо, що якобіан відображення $f$ в точці $z\in D$ можна
порахувати за допомогою рівності
$$J(z,
f)=|f_z|^2-|f_{\overline{z}}|^2\,.$$
Нехай $\mu=\mu(z, w): D\times {\Bbb C}\rightarrow {\Bbb D}$ і
$\nu=\nu(z, w): D\times {\Bbb C}\rightarrow {\Bbb D}$ функції такі,
що при кожному фіксованому $w\in {\Bbb C}$ виконано умову $|\mu(z,
w)|+|\nu(z, w)|<1$ при майже всіх $z\in D.$ Розглянемо {\it
квазілінійне рівняння Бельтрамі з двома характеристиками:}
\begin{equation}\label{eq1:}
f_{\overline{z}}=\mu(z, f(z))\cdot f_z+\nu(z, f(z))\cdot
\overline{f_z}\,.
\end{equation}
Функція $K_{\mu, \nu}(z,w)$ в~(\ref{eq1}) називається {\it
дилатацією рівняння~(\ref{eq1:})}. Відображення $f:D\rightarrow
{\Bbb C}$ будемо називати {\it регулярним розв'язком
рівняння~(\ref{eq1:}),} якщо $f\in W_{\rm loc}^{1,1}$ і $J(z, f)\ne
0$ майже скрізь у $D.$ Нехай $D$ -- область в ${\Bbb R}^n,$
$n\geqslant 2.$ Будемо говорити, що функція
${\varphi}:D\rightarrow{\Bbb R},$ що є локально інтегровною в
деякому околі точки $x_0\in D,$ має {\it скінченне середнє
коливання} в точці $x_0$ (пишемо: $\varphi\in FMO(x_0)$), якщо
\begin{equation}\label{eq17:}
{\limsup\limits_{\varepsilon\rightarrow
0}}\frac{1}{\Omega_n\varepsilon^n}\int\limits_{B(
x_0,\,\varepsilon)}
|{\varphi}(x)-\overline{{\varphi}}_{\varepsilon}|\ dm(x)\, <\,
\infty\,,
\end{equation}
де $\Omega_n$ -- об'єм одиничної кулі в ${\Bbb R}^n,$
$\overline{{\varphi}}_{\varepsilon}=\frac{1}{\Omega_n\varepsilon^n}\int\limits_{B(
x_0,\,\varepsilon)} {\varphi}(x)\ dm(x)$ (див., напр.,
\cite[розд.~2]{RSY$_2$}). Позначимо через $q_{z_0}(r)$ середнє
значення функції $Q$ над колом $$S(z_0, r)=\{|z-z_0|=r\}\,,$$
\begin{equation}\label{eq2}
q_{z_0}(r)=\frac{1}{2\pi}\int\limits_{0}^{2\pi}Q(z_0+re^{\,i\theta})\,d\theta\,.
\end{equation}
Є правильним наступне твердження (див. також відповідний результат,
встановлений для виродженого квазілінійного рівняння Бельтрамі з
однією характеристикою у \cite[теореми 1--2]{Sev}).
\medskip
\begin{theorem}\label{th4.6.1}{\sl\,
Нехай функції $\mu(z,w): {\Bbb D}\times{\Bbb C}\rightarrow {\Bbb D}$
і $\nu(z,w): {\Bbb D}\times{\Bbb C}\rightarrow {\Bbb D}$
задовольняють умови Каратеодорі і, крім того, $|\mu(z, w)|+|\nu(z,
w)|<1$ при всіх $w\in {\Bbb C}$ і майже всіх $z\in {\Bbb D}.$
Припустимо, що існує функція $Q:{\Bbb D}\rightarrow [1, \infty]$
така, що $K_{\mu, \nu}(z, w)\leqslant Q(z)\in L_{\rm loc}^1({\Bbb
D})$ для майже всіх $z\in {\Bbb D}$ і всіх $w\in {\Bbb C},$ де
$K_{\mu, \nu}$ визначено в (\ref{eq1}). Припустимо, що $Q\in
FMO({\Bbb D}),$ або для кожного $z_0\in {\Bbb D}$ виконано умову
\begin{equation}\label{eq15:}
\int\limits_{0}^{\delta(z_0)}\frac{dr}{rq_{z_0}(r)}=\infty\,,
\end{equation}
де $\delta(z_0)$ -- деяке додатне число,
$\delta(z_0)<{\rm dist\,}(z_0,\partial {\Bbb D}),$ а $q_{z_0}(r)$
визначено в~(\ref{eq2}).
Тоді рівняння (\ref{eq1:}) має регулярний гомеоморфний розв'язок $f$
класу $W_{\rm loc}^{1,1}$ в ${\Bbb D},$ такий що $f^{\,-1}\in W_{\rm
loc}^{\,1,2}(f({\Bbb D})).$ }
\end{theorem}

\medskip
Іншого роду результати стосуються випадку, коли
рівняння~(\ref{eq1:}) має лише неперервний, але логарифмічно
гельдеревий розв'язок, див. роботи \cite{Sev$_1$}--\cite{SevSkv} з
приводу аналогічних лінійних рівнянь з одним коефіцієнтом $\mu.$
Нехай $J(z, f)\ne 0$ і нехай відображення $f$ має частинні похідні
$f_z$ і $f_{\overline{z}}$ у точці $z.$ Тоді {\it максимальною
дилатацією відображення} $f$ в точці $z$ будемо називати наступну
функцію:
\begin{equation}\label{eq1D}
K_{\mu_f}(z)=\frac{|f_z|+|f_{\overline{z}}|}{|f_z|-|f_{\overline{z}}|}\,.
\end{equation}
Покладемо $K_{\mu_f}(z)=1$ в точках $z,$ де
$|f_z|+|f_{\overline{z}}|=0$ та $K_{\mu_f}(z)=\infty$ у точках $z,$
де $|f_z|+|f_{\overline{z}}|\ne 0,$ але $J(z, f)=0.$ Визначимо також
{\it внутрішню дилатацію порядку $p\geqslant 1 $} відображення $f$
за допомогою співвідношення
\begin{equation}\label{eq17}
K_{I, p}(z,
f)=\frac{{|f_z|}^2-{|f_{\overline{z}}|}^2}{{(|f_z|-|f_{\overline{z}}|)}^p}
\end{equation}
де $J(z, f)\ne 0.$ Як і у (\ref{eq1D}), покладемо $K_{I, p}(z)=1,$
якщо $|f_z|+|f_{\overline{z}}|=0,$ і $K_{I, p}(z)=\infty$ у точках,
де $J(z, f)=0,$ проте $|f_z|+|f_{\overline{z}}|\ne 0.$ Зауважимо, що
$K_{I, 2}(z)=K_{\mu}(z).$ Покладемо $\Vert
f^{\,\prime}(z)\Vert=|f_z|+|f_{\overline{z}}|.$  Нагадаємо, що
гомеоморфізм $f:D\rightarrow {\Bbb C}$ називається {\it
квазіконформним,} якщо $f\in W_{\rm loc}^{1, 2}(D)$ і, крім того,
існує стала $K\geqslant 1$ така, що $\Vert
f^{\,\prime}(z)\Vert^2\leqslant K\cdot |J(z, f)|$ майже скрізь.

\medskip
Нехай $\mu=\mu(z, w):{\Bbb D}\times {\Bbb C}\rightarrow {\Bbb D}$ і
$\nu=\nu(z, w): {\Bbb D}\times {\Bbb C}\rightarrow {\Bbb D}$
функції, для яких існує вимірна за Лебегом функція $q:{\Bbb
D}\rightarrow [0, 1)$ така, що при кожному фіксованому $w\in {\Bbb
C}$ і майже всіх $z\in {\Bbb D}$ виконано умову $|\mu(z, w)|+|\nu(z,
w)|\leqslant q(z)< 1.$ Зафіксуємо $n\geqslant 1$ і покладемо
\begin{equation}\label{eq12:} \mu_n(z, w)= \left
\{\begin{array}{rr}
 \mu(z, w),&  Q_0(z)\leqslant n,
\\ 0\ , & Q_0(z)> n\,.
\end{array} \right.
\end{equation}
і
\begin{equation}\label{eq12:A} \nu_n(z, w)= \left
\{\begin{array}{rr}
 \nu(z, w),& Q_0(z)\leqslant n,
\\ 0\ , & Q_0(z)> n\,,
\end{array} \right.
\end{equation}
де $Q_0$ визначено рівністю
\begin{equation}\label{eq10A}
Q_0(z)=\quad\frac{1+q(z)}{1-q(z)}\,.
\end{equation}
Нехай $f_n$ -- гомеоморфний $ACL$-розв'язок рівняння~(\ref{eq1:}), в
якому покладемо $\mu\mapsto\mu_n(z, w),$ $\nu\mapsto\nu_n(z, w),$
такий що $f_n(0)=0,$ $f_n(1)=1$ (вказаний розв'язок визначено
коректно з огляду на \cite[теорема~8.2]{Bo}). Зауважимо, що
відображення $g_n=f^{\,-1}_n$ є квазіконформним; зокрема, воно є
диференційовним майже скрізь. Нехай $K_{\mu_{g_n}}(w)$ -- дилатація
оберненого відображення $g_n,$ тобто,
\begin{equation}\label{eq3D}
K_{\mu_{g_n}}(w)=\quad
\frac{{|(g_n)_w|}^2-{|(g_n)_{\overline{w}}|}^2}{{(|(g_n)_w|-|(g_n)_{\overline{w}}|)}^2}\,.
\end{equation}
Визначимо також {\it внутрішню дилатацію порядку $p$ відображення
$g_n$ в точці $w$} за допомогою рівності
\begin{equation}\label{eq18}
K_{I, p}(w,
g_n)=\frac{{|(g_n)_w|}^2-{|(g_n)_{\overline{w}}|}^2}{{(|(g_n)_w|-|(g_n)_{\overline{w}}|)}^p}\,.
\end{equation}

\medskip
Виконується наступне твердження.

\medskip
\begin{theorem}\label{th1A}{\sl\, Нехай $\mu,$ $\nu,$ $\mu_n,$
$\nu_n,$ $f_n$ і $g_n$ такі, як вказано вище. Нехай $Q:{\Bbb
D}\rightarrow[1, \infty]$ -- вимірна за Лебегом функція. Припустимо,
що виконуються наступні умови:

\medskip
1) для кожних $0<r_1<r_2<1$ і $y_0\in {\Bbb D}$ існує множина
$E\subset[r_1, r_2]$ додатної лебегової міри така, що функція $Q$ є
інтегровною по колах $S(y_0, r)$ для кожного $r\in E;$

\medskip
2) знайдеться число $1<p\leqslant 2$ і стала $M>0$ такі, що
\begin{equation}\label{eq10B}
\int\limits_{\Bbb D}K_{I, p}(w, g_n)\,dm(w)\leqslant M
\end{equation}
для всіх $n=1,2,\ldots ,$ де $K_{I, p}(w, g_n)$ визначено
у~(\ref{eq18});

\medskip
3) Нерівність
\begin{equation}\label{eq10C}
K_{\mu_{g_n}}(w)\leqslant Q(w)
\end{equation}
виконується для майже всіх $w\in {\Bbb D},$ де $K_{\mu_{g_n}}$
визначено в~(\ref{eq3D}).

Тоді рівняння~(\ref{eq1:}) має неперервний $W_{\rm loc}^{1, p}({\Bbb
D})$-розв'язок $f$ в ${\Bbb D}.$} \end{theorem}

\medskip
\begin{corollary}\label{cor5}
{\sl\, Зокрема, твердження теореми~\ref{th1A} виконується, якщо в
цій теоремі ви відмовляємося від умови~1), вимагаємо умову~3), а
умову 2) заміняємо наступною: $Q\in L^1({\Bbb D}).$ В цьому випадку,
розв'язок $f$ рівняння~(\ref{eq1:}) може бути обраним таким, що
рівність
\begin{equation}\label{eq11A}
|f(x)-f(y)|\leqslant\frac{C\cdot (\Vert
Q\Vert_1)^{1/2}}{\log^{1/2}\left(1+\frac{r_0}{2|x-y|}\right)}
\end{equation}
виконується для довільного компакту $K\subset {\Bbb D}$ і всяких $x,
y\in K,$ де $\Vert Q\Vert_1$ позначає $L^1$-норму функції $Q$ в
${\Bbb D},$ $C>0$ деяка стала і $r_0=d(K,
\partial {\Bbb D}).$ Якщо додатково $Q(z)\in FMO({\Bbb D}),$ або виконано умову~(\ref{eq15:}),
то відображення $f$ можна обрати гомеоморфізмом.}
\end{corollary}

\medskip
{\bf 2. Існування гомеоморфного розв'язку квазілінійного рівняння в
одиночному крузі.} Для зручності покладемо $\partial f=f_z,$
$\overline{\partial}f=f_{\overline{z}}.$ Наступне твердження
доведено в~\cite[лема~9.1]{GRSY}.

\medskip
\begin{proposition}\label{pr4.6.2}{\sl\,
Нехай $D\subset {\Bbb C},$ і нехай $f_n:D\rightarrow {\Bbb C}$ --
послідовність гомеоморфних розв'язків рівняння $\overline{\partial}
f_n=\mu_n(z)\partial f_n+ \nu_n(z)\overline{\partial f_n}$ класу
$W_{\rm loc}^{1, 1}$ таких, що
$$\frac{1+|\mu_n(z)|+|\nu_n(z)|}{1-|\mu_n(z)|-|\nu_n(z)|}\quad\leqslant\quad
Q(z)\,\in\,L_{loc}^1(D)$$
при всіх $n=1,2,\ldots.$ Якщо $f_n\rightarrow f$ локально рівномірно
в $D$ при $n\rightarrow \infty$ і $f:D\rightarrow {\Bbb C}$ --
гомеоморфізм у $D,$ то $f\in W_{\rm loc}^{1, 1}$ і, крім того,
$\partial f_n$ і $\overline{\partial}f_n$ збігаються слабко в
$L_{loc}^1$ до $\partial f$ і $\overline{\partial}f,$ відповідно.
Якщо $\mu_n\rightarrow\mu$ при $n\rightarrow\infty$ і
$\nu_n\rightarrow\nu$ при $n\rightarrow\infty$ майже скрізь, то
$\overline{\partial} f=\mu(z)\partial f+ \nu(z)\partial f$ майже
скрізь.}
\end{proposition}

\medskip
Ключовим твердженням даного розділу є наступна лема, аналоги якої
неодноразово доводились в різних ситуаціях (див., напр.,
\cite[лема~9.1]{GRSY}, \cite[лема~1]{Sev}).

\medskip
\begin{lemma}\label{lem4.6.1}{\sl\,
Нехай функції $\mu(z,w): {\Bbb D}\times{\Bbb C}\rightarrow {\Bbb D}$
і $\nu(z,w): {\Bbb D}\times{\Bbb C}\rightarrow {\Bbb D}$
задовольняють умови Каратеодорі і, крім того, $|\mu(z, w)|+|\nu(z,
w)|<1$ при всіх $w\in {\Bbb C}$ і майже всіх $z\in {\Bbb D}.$ Нехай,
крім того, існує функція $Q:{\Bbb D}\rightarrow [1, \infty]$ така,
що $K_{\mu, \nu}(z, w)\leqslant Q(z)\in L_{\rm loc}^1({\Bbb D})$ для
майже всіх $z\in {\Bbb D}$ і всіх $w\in {\Bbb C},$ де $K_{\mu, \nu}$
визначено в (\ref{eq1}). Припустимо, що для будь-якого $z_0\in {\Bbb
D}$ існують $0<\varepsilon^{\,\prime}_0\leqslant \varepsilon_0< {\rm
dist\,}(z_0,
\partial {\Bbb D}),$ $c>0,$ $0<p\leqslant 2$ і вимірна за Лебегом функція $\psi:(0, \infty)\rightarrow (0,
\infty)$ такі, що
\begin{equation}\label{eq1A} 0<I(\varepsilon,
\varepsilon_0):=\int\limits_{\varepsilon}^{\varepsilon_0}\psi(t)dt <
\infty
\end{equation}
при $\varepsilon \in(0, \varepsilon^{\,\prime}_0),$
$I(\varepsilon)\rightarrow \infty$ при $\varepsilon\rightarrow 0,$
і, при цьому,
\begin{equation}\label{eq10:}
\int\limits_{\varepsilon<|z-z_0|<\varepsilon_0}Q(z)\cdot\psi^{\,2}(|z-z_0|)
 \, dm(z)\leqslant c\cdot I^{\,p}(\varepsilon, \varepsilon_0)\,.
\end{equation}
Тоді рівняння (\ref{eq1:}) має регулярний гомеоморфний розв'язок $f$
класу $W_{\rm loc}^{1,1}$ в ${\Bbb D},$ такий що $f^{\,-1}\in W_{\rm
loc}^{\,1,2}(f({\Bbb D})).$}
\end{lemma}

\begin{proof} Переважно будемо користуватися схемою
доведення леми~1 в~\cite{Sev} з урахуванням відмінностей між різними
типами рівнянь Бельтрамі, зазначених при доведенні леми~9.1
в~\cite{GRSY}. Розглянемо послідовності функцій
\begin{equation}\label{eq12:B} \mu_n(z, w)= \left
\{\begin{array}{rr}
 \mu(z, w),& Q(z)\leqslant n,
\\ 0\ , & Q(z)>n
\end{array} \right.
\end{equation}
і
\begin{equation}\label{eq12:AB} \nu_n(z, w)= \left
\{\begin{array}{rr}
 \nu(z, w),& Q(z)\leqslant n,
\\ 0\ , & Q(z)>n\,.
\end{array} \right.
\end{equation}
Зауважимо, що $K_{\mu_n, \nu_n}(z,w)\leqslant n$ при майже всіх
$z\in В$ і всіх $w\in{\Bbb C}.$ Отже,
$$|\mu_n (z, w)|+|\nu_n (z, w)|\leqslant\frac{n-1}{n+1}<1$$
при всіх $w\in {\Bbb C}$ і майже всіх $w\in {\Bbb C}.$ Отже, за
\cite[теорема~8.2]{Bo} існує $n$-квазіконформний розв'язок $f_n$
рівняння
$$\overline{\partial}f_n=\partial f_n\mu_n(z,
f_n)+\overline{\partial f_n}\nu_n(z, f_n)$$
такий, що $f_n(0)=0,$ $f_n(1)=1.$ З огляду на означення відображень
$\mu_n$ і $\nu_n$ маємо: $K_{\mu_{f_n}}(z)\leqslant Q(z)$ майже
скрізь. В такому випадку, за \cite[співвідношення~(6.6), гл.~V]{LV}
кожне $f_n$ задовольняє оцінку
\begin{equation}\label{eq4:}
M(f_n(\Gamma(S(z_0, r_1), S(z_0, r_2),\,A)))\leqslant
\int\limits_{A(z_0, r_1, r_2)} Q(z)\cdot \eta^2(|z-z_0|) dm(z)
\end{equation}
у будь-якому кільці $A=A(z_0, r_1,r_2)=\{z\in {\Bbb C}:
r_1<|z-z_0|<r_2\}$ при довільному $z_0\in {\Bbb D},$ довільних
$0<r_1<r_2<{\rm dist}\,(z_0,
\partial {\Bbb D})$ і кожної вимірної за Лебегом функції $\eta:
(r_1,r_2)\rightarrow [0,\infty]$ такої, що
\begin{equation}\label{eq5:}
\int\limits_{r_1}^{r_2}\eta(r)\, dr \geqslant 1\,.
\end{equation}
Тоді за пропозицією 2 і зауваженням 2 в \cite{Sev} послідовність
$f_n$ є одностайно неперервною відносно хоральної (сферичної)
метрики $h$ в ${\Bbb R}^n$ (див. означення~12.1 у \cite{Va}). За
критерієм Арцела-Асколі (див., напр., \cite[теорема~20.4]{Va}) існує
підпослідовність $f_{n_k}$ послідовності $f_n,$ $k=1,2,\ldots ,$ яка
збігається при $k\rightarrow\infty$ до декого відображення $f$
локально рівномірно в ${\Bbb D}.$ Отже, за лемою~4.2 в~\cite{RSS}
відображення $f$ є або гомеоморфізмом у ${\Bbb D},$ або сталою в
$\overline{\Bbb R}^n.$ Друга ситуація виключена враховуючи умови
нормування $f_n(0)=0,$ $f_n(1)=1.$ Зауважимо, що для майже всіх
$z\in{\Bbb D}$ знайдеться номер $k_0=k_0(z)$ такий, що
$\mu_{n_k}(z,w)=\mu(z,w),$ $\nu_{n_k}(z,w)=\nu(z,w)$ при
$n_k\geqslant n_{k_0}(z)$ і всіх $w\in{\Bbb C}.$ Отже, для м.в. $z,$
$$\mu_{n_k}(z)= \mu_{n_k}(z, f_{n_k}(z))\rightarrow \mu(z,
f(z))\,,$$
$$\nu_{n_k}(z)= \nu_{n_k}(z, f_{n_k}(z))\rightarrow \nu(z,
f(z))$$
при $k\rightarrow\infty.$ За твердженням~\ref{pr4.6.2}
$\overline{\partial} f=\mu(z, f)\partial f+ \nu(z, f)\partial f,$
тобто, $f$ -- гомеоморфний розв'язок рівняння~(\ref{eq1:}), причому
$f\in W_{\rm loc}^{1, 1}.$

\medskip
Зауважимо, що за теоремою збіжності розв'язків рівняння Бельтрамі
$f^{\,-1}\in W^{1,2}_{\rm loc}$ (див. \cite[наслідок~2.4]{GRSY}).
Тоді за теоремою Малого-Мартіо $f^{\,-1}$ має $N$-властивість Лузіна
(див., напр., \cite[наслідок~B]{MM}). Нарешті, за теоремою
Пономарьова $J(z, f)\ne 0$ майже скрізь, див. \cite[теорема~1]{Pon}.
Лему доведено. \end{proof} $\Box$

\medskip
{\it Доведення теореми~\ref{th4.6.1}} випливає з леми~\ref{lem4.6.1}
і \cite[лема~5.1]{Sev$_2$}.

\medskip
{\bf 3. Існування неперервного розв'язку.} Аналого наступної леми
доведений в~\cite[теорема~9.1]{GRSY}, див. також
\cite[лема~5.1]{SevSkv}.

\medskip
\begin{lemma}\label{lem1}
{\sl\, Нехай $1<p\leqslant 2,$ нехай $\mu:D\rightarrow {\Bbb D}$ --
вимірна за Лебегом функція, і нехай $f_k,$ $k=1,2,\ldots $ --
послідовність гомеоморфізмів, що зберігають орієнтацію, області $D$
на себе, які належать класу $W_{\rm loc}^{1, 2}(D)$ і задовольняють
рівняння
\begin{equation}\label{eq1B}
\overline{\partial}f_n=\partial f_n\mu_n(z)+\overline{\partial
f_n}\nu_n(z)\,,
\end{equation}
де $\mu_n,$ $\nu_n$ -- вимірні за Лебегом функції, які задовольняють
нерівність $|\nu_n(z)|+|\mu_n(z)|<1$ майже скрізь. Припустимо, що
$f_n$ збігається локально рівномірно в $D$ до відображення
$f:D\rightarrow {\Bbb C},$ а послідовності $\mu_n(z)$ та $\nu_n(z)$
збігаються до $\mu(z)$ і $\nu(z),$ відповідно, при
$n\rightarrow\infty$ майже скрізь. Нехай також обернені відображення
$g_n:=f_n^{\,-1}$ належать класу $W_{\rm loc}^{1, 2}(D),$ при цьому,
при майже всіх $w\in D$
$$\int\limits_{D}K_{I, p}(w, g_k)\,dm(w)\leqslant M$$
для деякого $M>0$ і кожного $n=1,2,\ldots .$

Тоді $f\in W_{\rm loc}^{1, p}(D)$ і $\mu,$ $\nu$  -- комплексні
характеристики відображення $f,$ тобто,
$\overline{\partial}f=\partial f\mu(z)+\overline{\partial f}\nu(z)$
при майже всіх $z\in D.$
 }
\end{lemma}

\medskip
\begin{proof}
Будемо в цілому дотримуватись схеми, викладеної при
доведенні~\cite[теорема~9.1]{GRSY}, див. також
\cite[лема~5.1]{SevSkv}. Позначимо $\partial f=f_z$ і
$\overline{\partial}f=f_{\overline{z}}.$ Нехай $C$ -- довільний
компакт в $D.$ Оскільки за умовою відображення~$g_n=f_n^{\,-1}$
належать класу $W_{\rm loc}^{1, 2},$ то $g_n$ мають $N$-властивість
Лузіна, див., напр., \cite[наслідок~B]{MM}. Тоді якобіан $J(z, f_n)$
майже скрізь не дорівнює нулю, див, напр., \cite[теорема~1]{Pon},
більше того, має місце заміна змінних в інтегралі, див,
напр.,~\cite[теорема~3.2.5]{Fe}. У такому випадку, будемо мати:
$$\int\limits_{C}{\Vert f^{\,\prime}_n(z)\Vert}^p\,dm(z)=
\int\limits_C \frac{{\Vert f^{\,\prime}_n(z)\Vert}^p}{J(z,
f_n)}\cdot J(z, f_n)\,dm(z)=$$
\begin{equation}\label{eq4D}
=\int\limits_{f_n(C)}K_{I, p}(w, g_n)\, dm(w)\leqslant M<\infty\,.
\end{equation}
З~(\ref{eq4D}) випливає, що $f\in W_{\rm loc}^{1, p}$ і, крім того,
$\partial f_n$ і $\overline{\partial} f_n$ слабко збігаються в
$L_{\rm loc}^1(D)$ до $\partial f$ і $\overline{\partial} f,$
відповідно (див.~\cite[лема~III.3.5]{Re$_2$}; див. також
\cite[лема~2.1]{RSY$_3$}).

\medskip
Залишилось показати, що  відображення $f$ є розв'язком рівняння
Бельтрамі $f_{\overline{z}}=\mu(z)\cdot f_z.$ Покладемо
$\zeta(z)=\overline{\partial} f(z)-\mu(z)\partial f(z)-\overline{\nu
\partial f(z)}$ і покажемо, що $\zeta(z)=0$ майже скрізь. Нехай $B$
-- довільний круг, що лежить разом зі своїм замиканням в $D.$ За
нерівністю трикутника
\begin{equation}\label{eq9}
\left|\int\limits_B\zeta(z)\,dm(z)\right|\leqslant
I_1(n)+I_2(n)+I_3(n)\,,
\end{equation}
де
\begin{equation}\label{eq7}
I_1(n)=\left|\int\limits_B(\overline{\partial}f(z)-\overline{\partial}f_n(z))\,dm(z)\right|\,,
\end{equation}
\begin{equation}\label{eq8}
I_2(n)=\left|\int\limits_B(\mu(z)\partial f(z)-\mu_n(z)\partial
f_n(z))\,dm(z)\right|
\end{equation}
і
\begin{equation}\label{eq9C}
I_3(n)=\left|\int\limits_B(\nu(z)\overline{\partial
f(z)}-\nu_n(z)\overline{\partial f_n(z)})\,dm(z)\right|\,.
\end{equation}
З огляду на доведене вище, $I_1(n)\rightarrow 0$ при
$n\rightarrow\infty.$ Залишилось розібратися з виразами для $I_2(n)$
і $I_3(n).$ Для цього зауважимо, що за нерівністю трикутника
$I_2(n)\leqslant I^{\,\prime}_2(n)+I^{\,\prime\prime}_2(n),$ де
$$I^{\,\prime}_2(n)=
\left|\int\limits_B\mu(z)(\partial f(z)-\partial
f_n(z))\,dm(z)\right|$$
і
$$I^{\,\prime\prime}_2(n)=
\left|\int\limits_B(\mu(z)-\mu_n(z))\partial
f_n(z)\,dm(z)\right|\,.$$
З огляду на слабку збіжність $\partial f_n\rightarrow
\partial f$ в $L^1_{\rm loc}(D)$ при $n\rightarrow\infty,$ ми
отримаємо, що $I^{\,\prime}_2(n)\rightarrow 0$ при
$n\rightarrow\infty,$ оскільки $\mu\in L^{\infty}(D).$ Більше того,
оскільки за доведеним вище відображення $\partial f$ інтегровне з
квадратом, має місце абсолютна неперервність в інтегралі
$\int\limits_E|\partial f(z)|\,dm(z).$ Крім того, оскільки $\partial
f_n\rightarrow \partial f$ слабко в $L^1_{\rm loc}(D),$ то для
заданого $\varepsilon>0$ знайдеться $\delta=\delta(\varepsilon)>0$
таке, що
$$\int\limits_E|\partial f_n(z)|\,dm(z)\leqslant$$
\begin{equation}\label{eq5}\leqslant
\int\limits_E|\partial f_n(z)-\partial f(z)|\,dm(z)+
\int\limits_E|\partial f(z)|\,dm(z)<\varepsilon\,,
\end{equation}
як тільки $m(E)<\delta,$ $E\subset B,$ і номера $n$ є достатньо
великими.

\medskip
Остаточно, за теоремою Єгорова (див.~\cite[теорема~III.6.12]{Sa})
для кожного $\delta>0$ знайдеться множина $S\subset B$ така, що
$m(B\setminus S)<\delta$ і $\mu_n(z)\rightarrow \mu(z)$ рівномірно
на~$S.$ Тоді $|\mu_n(z)-\mu(z)|<\varepsilon$ при всіх $n\geqslant
n_0,$ деякому $n_0=n_0(\varepsilon)$ і всіх $z\in S.$ З огляду
на~(\ref{eq5}), а також з огляду на~(\ref{eq4D}) і нерівність
Гельдера, маємо, що
$$I^{\,\prime\prime}_2(n)\leqslant \varepsilon \int\limits_S|\partial
f_n(z)|\,dm(z)+ 2\int\limits_{B\setminus S}|\partial
f_n(z)|\,dm(z)<$$
\begin{equation}\label{eq6}
<\varepsilon\cdot\left\{\left(\int\limits_D K_{I, p}(w, g_n)\,
dm(w)\right)^{1/p}\cdot (m(B))^{(p-1)/p}+2\right\}\leqslant
\end{equation}
$$\leqslant \varepsilon\cdot\left\{M^{1/p}\cdot (m(B))^{(p-1)/p}+2\right\}$$
при тих же $n\geqslant n_0.$ Те, що
\begin{equation}\label{eq9A}
I_3(n)\rightarrow 0
\end{equation}
при $n\rightarrow\infty,$ може бути доведено аналогічно. Отже,
з~(\ref{eq7}), (\ref{eq8}), (\ref{eq9C}), (\ref{eq6}) і (\ref{eq9A})
випливає, що $\int\limits_B\zeta(z)\,dm(z)=0$ для всіх кругів $B,$
компактно вкладених в $D.$ За теоремою Лебега про диференціювання
невизначеного інтеграла (див. \cite[IV(6.3)]{Sa}) маємо, що
$\zeta(z)=0$ майже скрізь в $D.$ Лему доведено.~$\Box$
\end{proof}

\medskip
{\it Доведення теореми~\ref{th1A}.} Розглянемо по\-слі\-дов\-ність
комплекснозначних функцій
\begin{equation}\label{eq12:C} \mu_n(z, w)= \left
\{\begin{array}{rr}
 \mu(z, w),&  Q_0(z)\leqslant n,
\\ 0\ , & Q_0(z)> n
\end{array} \right.
\end{equation}
і
\begin{equation}\label{eq12:D} \nu_n(z, w)= \left
\{\begin{array}{rr}
 \nu(z, w),& Q_0(z)\leqslant n,
\\ 0\ , & Q_0(z)> n\,,
\end{array} \right.
\end{equation}
де $Q_0(z)$ визначається співвідношенням~(\ref{eq10A}). Нехай $f_n$
-- гомеоморфний $ACL$-розв'язок рівняння~(\ref{eq1:}), в якому
покладемо $\mu\mapsto\mu_n(z, w),$ $\nu\mapsto\nu_n(z, w),$ такий що
$f_n(0)=0,$ $f_n(1)=1$ (вказаний розв'язок визначено коректно з
огляду на \cite[теорема~8.2]{Bo}). Зауважимо, що відображення
$g_n=f^{\,-1}_n$ є квазіконформним; зокрема, воно є диференційовним
майже скрізь і належить класу~$W_{\rm loc}^{1, 2}({\Bbb D}).$
За~\cite[теорема~6.10]{MRSY$_1$} і з огляду на умову~(\ref{eq10C})
для кожного $n\in {\Bbb N}$
\begin{equation} \label{eq2*B}
M(g_n(\Gamma))\leqslant \int\limits_{\Bbb
D}K_{\mu_{g_n}}(w)\cdot\rho_*^2 (w) \,dm(w)\leqslant
\int\limits_{\Bbb D}Q(w)\cdot\rho_*^2 (w) \,dm(w)
\end{equation}
для довільної сім'ї кривих $\Gamma$ в ${\Bbb D}$ і кожної функції
$\rho_*\in {\rm adm}\,\Gamma,$ де $M$ -- модуль сім'ї кривих (див.,
напр.,~\cite[розд.~6]{Va}). За~\cite[теорема~1.1]{SevSkv} сім'я
відображень $f_n$ одностайно неперервна в ${\Bbb D}.$ Отже, з огляду
на теорему Арцела-Асколі $f_n$ є нормальною сім'єю
(див.~\cite[теорема~20.4]{Va}), іншими словами, знайдеться
підпослідовність $f_{n_l}$ послідовності $f_n,$ що збігається
локально рівномірно в ${\Bbb D}$ до деякого відображення $f:{\Bbb
D}\rightarrow \overline{{\Bbb D}}.$  Зауважимо, що кожне
відображення $f_n$ задовольняє рівняння
$$\overline{\partial} f_n=\mu_n(z, f_n(z))\partial f_n+
\nu_n(z, f_n(z))\overline{\partial f_n}\,.$$
Оскільки $Q_0(z)$ скінченна майже скрізь, для майже всіх $z\in{\Bbb
D}$ знайдеться номер $l_0=l_0(z)$ такий, що
$\mu_{n_l}(z,w)=\mu(z,w),$ $\nu_{n_l}(z,w)=\nu(z,w)$ при $l\geqslant
l_0$ і всіх $w\in{\Bbb C}.$ Оскільки за припущенням функції $\mu$ і
$\nu$ задовольняють умову Каратеодорі, будемо мати:
$$\mu_{n_l}(z, f_{n_l}(z))\rightarrow \mu(z, f(z))\,,$$
$$\nu_{n_l}(z, f_{n_l}(z))\rightarrow \nu(z, f(z))$$
при $l\rightarrow\infty.$ Тоді за лемою~\ref{lem1} відображення $f$
належить класу $W_{\rm loc}^{1, p}({\Bbb D})$ і є розв'язком
вихідного рівняння Бельтрамі~(\ref{eq1:}).~ $\Box$

\medskip
{\it Доведення наслідку~\ref{cor5}} відбувається аналогічно
доведенню наслідку~5.1 в~\cite{SevSkv}. Дійсно, за теоремою Фубіні з
умови $Q\in L^1({\Bbb D})$ випдливає вимірність інтегралів
$\int\limits_{S(x_0, r)\cap D}\,Q(x)\,d\mathcal{H}^1(x)$ як функцій
від $r$ так їх скінченність майже скрізь $0<r<\infty$ (див., напр.,
\cite[теорема~8.1.III]{Sa}). В цьому випадку, умови~(\ref{eq10B}) і
(\ref{eq10C}) виконуються при $p=2.$ Отже, існування розв'язку
рівняння~(\ref{eq1:}) і його належність класу~$W_{\rm loc}^{1,
2}({\Bbb D})$ безпосередньо випливають з теореми~\ref{th1A}.

\medskip
Крім того, за~\cite[теорема~1]{SSD}
$$|f_n(z)-f_n(z_0)|\leqslant\frac{C\cdot (\Vert
Q\Vert_1)^{1/2}}{\log^{1/2}\left(1+\frac{r_0}{|z-z_0|}\right)}\quad\forall\,\,z\in
B(z_0, r_0)$$
в довільній точці $z_0\in {\Bbb D},$ де $\Vert Q\Vert_1$ -- норма
$Q$ в $L^1({\Bbb D}),$ $C$ -- деяка стала і $0<2r_0<{\rm
dist}\,(z_0,
\partial {\Bbb D}).$ Переходячи тут до границі при
$n\rightarrow\infty,$ маємо співвідношення~(\ref{eq11A}).
Наслідок~\ref{cor5} встановлено.~$\Box$

\medskip
Припустимо тепер, що~$Q\in FMO({\Bbb D}),$ або виконується
співвідношення~(\ref{eq15:}). Тоді послідовність $g_n$ утворює
одностайно неперервну сім'ю відображень (див.~\cite[теореми~6.1 і
6.5]{RS}). Отже, з огляду на теорему Арцела-Асколі $g_n$ є
нормальною сім'єю (див.~\cite[теорема~20.4]{Va}), іншими словами,
знайдеться підпослідовність $g_{n_l}$ послідовності $g_n,$ що
збігається локально рівномірно в ${\Bbb D}$ до деякого відображення
$g:{\Bbb D}\rightarrow \overline{{\Bbb D}}.$ В силу умови нормування
$g_{n_l}(0)=0$ і $g_{n_l}(1)=1$ при всіх $l=1,2,\ldots .$ Тоді в
силу~\cite[теорема~4.1]{RSS} відображення $g$ є гомеоморфізмом в
${\Bbb D},$ крім того, за~\cite[лема~3.1]{RSS} ми маємо також, що
$f_{n_l}\rightarrow f=g^{\,-1}$ при $l\rightarrow\infty$ локально
рівномірно в ${\Bbb D}.$ Далі застосуємо схему міркувань,
використану вище у випадку інтегровної функції $Q.$ Оскільки
$\mu_n(z)\rightarrow \mu(z)$ при $n\rightarrow\infty$ і при майже
всіх $z\in {\Bbb D},$ за лемою~\ref{lem1} відображення $f$ належить
класу $W_{\rm loc}^{1, 2}({\Bbb D})$ і є розв'язком вихідного
квазілінійного рівняння Бельтрамі~(\ref{eq1:}). Наслідок
доведений.~$\Box$

\medskip
{\bf 4. Приклади.}

\medskip
\begin{example}\label{pr1}
Побудуємо приклад неперервного, але не гомеоморфного розв'язку
квазілінійного рівняння~(\ref{eq1:}), який задовольняє умови
наслідку~\ref{cor5} (зокрема, теореми~\ref{th1A}). Для спрощення
будемо розглядати ситуацію, коли $\nu(z, w)\equiv 0,$ причому,
користуємося конструкцію прикладу, наведеного
в~\cite[розділ~3]{Sev$_1$}. Нехай $p\geqslant 1$ -- довільне число і
нехай $0<\alpha<2/p.$ Як зазвичай, ми використовуємо запис
$z=re^{i\theta},$ $r\geqslant 0$ и $\theta\in [0, 2\pi).$ Покладемо
\begin{equation}\label{eq3A}\mu(z, w)= \left
\{\begin{array}{rr}
 e^{2i\theta}\cdot\frac{2r-\alpha(2r-1)}{2r+\alpha(2r-1)},& 1/2<|z|<1,
 |w|\geqslant 1,\\
e^{2i\theta}\cdot\frac{|w|^{\,\alpha}+1-\alpha(2r-1)}{|w|^{\,\alpha}+1+\alpha(2r-1)},&
1/2<|z|<1, |w|<1,
\\ 0\ , & |z|\leqslant 1/2\,.
\end{array} \right.
\end{equation}
Використовуючи співвідношення
$$\frac{\overline{\partial} f}{\partial f}=e^{2i\theta}\frac{rf_r+if_{\theta}}{rf_r-if_{\theta}}\,,$$
див. рівність~(11.129) в \cite{MRSY}, ми отримуємо, що відображення
\begin{equation}\label{eq4A}f(z)=\left
\{\begin{array}{rr}
 \frac{z}{|z|}(2|z|-1)^{1/\alpha},& 1/2<|z|<1,
\\ 0\ , & |z|\leqslant 1/2
\end{array} \right.
\end{equation}
є розв'язком рівняння~$f_{\overline{z}}=\mu(z, f(z))\cdot f_z,$ де
функція $\mu$ задається співвідношенням~(\ref{eq3A}). Зауважимо, що
існування розв'язку вказаного рівняння забезпечується
теоремою~\ref{th1A} (для цього перевіримо виконання всіх умов цієї
теореми). Дійсно, функція $\varphi(x, c)=\frac{x-c}{x+c}$ має
додатну похідну при $c:=\alpha(2r-1)>0,$ тому при $x\in [0, 1]$ ця
функція досягає свого максимального значення в точці $x=1.$ Отже,
\begin{equation}\label{eq15A}
\biggl|e^{2i\theta}\cdot\frac{|w|^{\,\alpha}+1-\alpha(2r-1)}{|w|^{\,\alpha}+1+\alpha(2r-1)}\biggr|=
\frac{|w|^{\,\alpha}+1-\alpha(2r-1)}{|w|^{\,\alpha}+1+\alpha(2r-1)}\leqslant
\frac{2-\alpha(2r-1)}{2+\alpha(2r-1)}\,.
\end{equation}
Беручи до уваги дріб праворуч у~(\ref{eq15A}), покладемо
\begin{equation}\label{eq16A}
\mu(z):=e^{2i\theta}\cdot \frac{2-\alpha(2r-1)}{2+\alpha(2r-1)}\,.
\end{equation}
Для заданої співвідношенням~(\ref{eq16A}) функції $\mu$ відповідною
їй максимальною дилатацією $K_{\mu}$ буде функція
\begin{equation}\label{eq5A}K_{\mu}(z)=\left
\{\begin{array}{rr}
 \frac{2}{\alpha(2|z|-1)},& 1/2<|z|<1,
\\ 1\ , & |z|\leqslant 1/2
\end{array} \right.\,.
\end{equation}
Нехай $k>1/\alpha.$ Зауважимо, що $K_{\mu}(z)\leqslant k$ при
$|z|\geqslant \frac{2+k\alpha}{2k\alpha}$ і $K_{\mu}(z)>k$ в іншому
випадку. Нехай, як і раніше,
\begin{equation*}\label{eq5B}\mu_k(z, w)= \left
\{\begin{array}{rr}
 \mu(z, w),& K_{\mu}(z)\leqslant k,
\\ 0\ , & K_{\mu}(z)> k\,.
\end{array} \right.
\end{equation*}
Відзначимо, що розв'язками рівняння $f_{\overline{z}}=\mu_k(z,
f(z))\cdot f_z$ є відображення
\begin{equation*}\label{eq6A}f_k(z)=\left
\{\begin{array}{rr}
 \frac{z}{|z|}(2|z|-1)^{1/\alpha},& \frac{2+k\alpha}{2k\alpha}<|z|<1,
\\ \frac{z}{\left(\frac{2+k\alpha}{2k\alpha}\right)
}\cdot{\left(\frac{2}{k\alpha}\right)}^{1/\alpha}\ , & |z|\leqslant
\frac{2+k\alpha}{2k\alpha}
\end{array} \right.\,,
\end{equation*}
при цьому, обернені відображення $g_k(y)=f_k^{\,-1}(y)$ обчислюються
за формулою
\begin{equation}\label{eq7A}g_k(y)=\left
\{\begin{array}{rr}
 \frac{y(|y|^{\alpha}+1)}{2|y|},& \left(\frac{2}{k\alpha}\right)^{1/\alpha}<|y|<1,
\\ \frac{y\cdot\frac{2+k\alpha}{2k\alpha}}
{\left(\frac{2}{k\alpha}\right)^{1/\alpha}} , &
|y|\leqslant\left(\frac{2}{k\alpha}\right)^{1/\alpha}
\end{array} \right.\,.
\end{equation}
З~(\ref{eq5A}) випливає, що
\begin{equation}\label{eq7B}K_{\mu_k}(z):=\frac{(f_k)_{\overline{z}}}{(f_k)_z}=\left
\{\begin{array}{rr}
 \frac{4|z|}{2\alpha(2|z|-1)},& \frac{2+k\alpha}{2k\alpha}<|z|<1,
\\ 1\ , & |z|< |z|\leqslant
\frac{2+k\alpha}{2k\alpha}
\end{array} \right.\,.
\end{equation}
Нам слід перевірити, чи виконується~(\ref{eq10C}) для деякої
інтегровної в ${\Bbb D}$ функції $Q.$ Для цієї мети, підставимо
відображення $g_k$ з~(\ref{eq7A}) у максимальну
дилатацію~$K_{\mu_k},$ визначену рівністю~(\ref{eq7B}). Тоді
\begin{equation*}\label{eq8A}K_{\mu_{g_k}}(y):=K_{\mu_k}(g_k(y))=\left
\{\begin{array}{rr}
 \frac{|y|^{\alpha}+1}{\alpha|y|^{\alpha}},&
 \left(\frac{2}{k\alpha}\right)^{1/\alpha}<|y|<1\,,
\\ 1\ , & |y|\leqslant\left(\frac{2}{k\alpha}\right)^{1/\alpha}
\end{array} \right.\,.
\end{equation*}
Зауважимо, що $K_{\mu_{g_k}}(y)\leqslant Q(y):=
\frac{|y|^{\alpha}+1}{\alpha|y|^{\alpha}}$ при всіх $y\in {\Bbb D},$
при цьому, функція $Q$ інтегровна в ${\Bbb D}$ навіть в степені $p,$
а не тільки в степені $1$ (див. міркування, використані при
розгляді~\cite[пропозиція~6.3]{MRSY}). За побудовою $f_k(0)=0$ і
$f_k(1)=1.$ Тому всі умови наслідку~\ref{cor5} (зокрема,
теореми~\ref{th1A}) виконуються, а у якості бажаного розв'язку
рівняння~$f_{\overline{z}}=\mu(z)\cdot f_z$ можна розглянути
відображення $f=f(z),$ визначене рівністю~(\ref{eq4A}). Більше того,
з доведення цієї теореми випливає, що відображення $f$ є вказаним
там розв'язком, оскільки $f$ є локально рівномірною границею
послідовності $f_k.$ Зауважимо, що відображення $f$ не є
гомеоморфним розв'язком, також воно не є ані відкритим, ані
дискретним.
\end{example}

\medskip
\begin{example}
Нехай $f$ -- відображення, визначене співвідношенням (\ref{eq4A}).
Користуючись формулою $f_z=f_rr_z+f_{\theta}\theta_z,$ ми отримаємо,
що
$f_z=(2r-1)^{(1/\alpha)-1}\left(\frac{2}{\alpha}+\frac{2r-1}{r^2}\right).$
Звідси випливає, що $f_z$ є дійсним числом при всіх $z\in {\Bbb D},$
отже, $f_z=\overline{f_z}.$ Отже, ми можемо записати:
$$f_{\overline{z}}=\mu(z, f)f_z= \frac12\cdot\mu(z, f)f_z+\frac12\cdot\nu(z, f)\overline{f_z}\,,$$
де $\mu(z, w)=\nu(z, w)$ і $\mu(z, w)$ визначено
співвідношенням~(\ref{eq3A}). За доведеним вище в прикладі~\ref{pr1}
рівняння $f_{\overline{z}}=\mu(z, f)f_{\overline{z}}+\nu(z,
f)\overline{f_z}$ має розв'язок $f,$ визначений
співвідношенням~(\ref{eq4A}), причому коефіцієнти цього рівняння
задовольняють умови наслідку~\ref{cor5} (зокрема,
теореми~\ref{th1A}).
\end{example}


КОНТАКТНА ІНФОРМАЦІЯ

\medskip
\noindent{{\bf Євген Олександрович Севостьянов} \\
{\bf 1.} Житомирський державний університет ім.\ І.~Франко\\
вул. Велика Бердичівська, 40 \\
м.~Житомир, Україна, 10 008 \\
{\bf 2.} Інститут прикладної математики і механіки
НАН України, \\
вул.~Добровольського, 1 \\
м.~Слов'янськ, Україна, 84 100\\
e-mail: esevostyanov2009@gmail.com}

\end{document}